\documentclass[11pt,letterpaper]{amsart}

\usepackage{fancyhdr}
\usepackage{bm}
\usepackage{hyperref}
\usepackage{graphicx} 
\usepackage{nicematrix}
\usepackage{amsthm,amssymb,amsmath}
\usepackage[T1]{fontenc}
\usepackage[utf8]{inputenc}
\usepackage{hyperref}
\usepackage{lipsum}
\usepackage{mathrsfs}
\usepackage{enumitem}
\usepackage{stmaryrd}
\usepackage{setspace}
\usepackage{yfonts}
\usepackage{float}
\usepackage{mathtools}
\usepackage{parskip}
\usepackage[all,cmtip]{xy}
\usepackage{tikz-cd}
\tikzcdset{row sep/normal=50pt, column sep/normal=50pt}
\newcommand{\floor}[1]{\lfloor #1 \rfloor}
\newtheorem{lemma}{Lemma}[section]

\newtheorem{theorem*}{Theorem}

\newtheorem{example*}[lemma]{Example}

\newtheorem{corollary}[lemma]{Corollary}

\newtheorem{manualtheoreminner}{Theorem}
\newenvironment{manualtheorem}[1]{%
  \renewcommand\themanualtheoreminner{#1}%
  \manualtheoreminner
}{\endmanualtheoreminner}
\usepackage{blindtext}

\usepackage[backend=biber, style=numeric, sorting=nyt]{biblatex}
\title{Singularity of $\mathbb{Q}$-divisors of multidegree one in multiprojective space}
\author{\small{Supravat Sarkar}}
\date{}
\begin{document}

\begin{abstract}
    We study singularity of effective $\mathbb{Q}$-divisors on products of projective spaces of multidegree $(1,1...,1).$ This generalizes works of Bath, Musta{\c{t}}{\u{a}} and Walther on singularity of square-free polynomials. We also give a lower bound on the log canonical threshold of a hypersurface in products of projective spaces.
\end{abstract}
\maketitle
\begin{center}
\textbf{Keywords}: rational singularity, log canonical threshold
\end{center}
\begin{center}
\textbf{MSC Number: 14B05, 14J17, 32S25} 
\end{center}

\section{Introduction}
 We work throughout over the field of complex numbers. A \textit{variety} is an integral finite type separated scheme over $\mathbb{C}$. Study of singularity of varieties is a ubiquitous topic in algebraic geometry. In this note, we are mainly concerned about the measure of singularities relevant to birational geometry and minimal model program. In birational geometry, one considers singularity of a pair $(X,\Delta)$, where $X$ is a variety, $\Delta$ is an effective $\mathbb{Q}$-divisor. There one classifies singularities in several classes like terminal, canonical, klt, plt, lc (see \cite{wilson2000birational}). By inversion of adjunction ({\cite[Theorem 7.5]{kollar1997singularities}}), one gets information about the singularity of a divisor $D$ on $X$ from the singularity of the pair $(X,D)$, and vice versa.

 In this note, we study singularities of effective $\mathbb{Q}$-divisors in products of projective spaces of multidegree $(1,1,...,1).$ Note that the line bundle associated to any effective Cartier divisor in $X=\prod_{i=1}^n\mathbb{P}^{m_i}$ is $\boxtimes_i \mathcal{O}_{\mathbb{P}_{m_i}}(a_i)$ for some nonnegative integers $a_i$. We call the vector $(a_1, a_2,..., a_n)$ the multidegree of $D$. More generally, if $D=\sum_j c_j D_j$ is a $\mathbb{Q}$-divisor on $X$, we define the multidegree of $D$ to be $\sum_j c_j$ deg $D_j$, where deg $D_j$ is the multidegree of $D_j$. Our main theorem is the following.
 \begin{manualtheorem}{A}\label{main}
     Let $n\geq1$, $m_1, m_2,..., m_n$ be positive integers, $X=\prod_{i=1}^n\mathbb{P}^{m_i}$. Let $\Delta\geq0$ be a $\mathbb{Q}$-divisor in $X$ of multidegree $(1,1,...,1).$ Then
     \begin{enumerate}
         \item[(i)] If  $\floor{\Delta}=0$, then $(X,\Delta)$ is klt.
         \item[(ii)] If  $\floor{\Delta}$ is irreducible, then $(X,\Delta)$ is plt.
     \end{enumerate}
 \end{manualtheorem}

 As an application, we show in Corollary \ref{rat} that irreducible divisors of multidegree $(1,1,...,1)$ in products of projective spaces have rational singularities. From this one easily recovers the main result of \cite{bath2024singularities}, as we show in Corollary \ref{squarefree}. In Corollary \ref{lct}, we give a lower bound of the log canonical threshold of an irreducible hypersurface in a product of projective spaces of multidegree $(d,d,..,d)$ for some $d\geq 1.$ This generalizes the $K=1$ case of {\cite[Theorem 1.3, Theorem 6.1]{cheltsov2001log}}.
\section{ Proof of the Main Theorem}
 First we prove two lemmas.
\begin{lemma}\label{normal}
Let $n\geq1$, $m_1, m_2,..., m_n$ be positive integers, $D$ an irreducible divisor in $X=\prod_{i=1}^n\mathbb{P}^{m_i}$ of multidegree $(1,1,...,1)$. Then $D$ is normal.
\end{lemma}
\begin{proof}
    For $n=1$, it is clear. So assume $n\geq 2.$ The argument is essentially the same as in {\cite[Claim 3.5, Theorem 1.1]{bath2024singularities}}. For $1\leq i\leq n$, let $\pi_i:X\to \prod_{j\neq i}\mathbb{P}^{m_i}$ be the projection. Let $W_i$ be the closed subset of points $x\in \prod_{j\neq i}\mathbb{P}^{m_i}$ such that $D\supset\pi_i^{-1}(x).$ Since $D$ has multidegree $(1,1,...,1)$, $\pi_i|_D$ is a $\mathbb{P}^{m_i-1}$ bundle outside $W_i$, hence $D$ is smooth and nonempty over the complement of $W_i$. Also, $\pi_i^{-1}(W_i)$ is a proper closed subset of $D$, so $\dim \pi_i^{-1}(W_i)\leq \sum_i m_i-2$.

    Since $D$ is a hypersurface in the smooth variety $X$, it is Cohen-Macaulay. By Serre's criterion of normality, it suffices to show that $D$ is regular in codimension $1$. Suppose not, let $Z$ be  an irreducible component of the singular locus of $D$ of dimension $\sum_i m_i-2.$ Since $D$ is smooth over the complement of $W_i$, we must have $Z\subset \pi_i^{-1}(W_i)$ for all $i$. But dim $Z=n-2\geq $ dim $\pi_i^{-1}(W_i)$, so $Z$ is an irreducible component of $\pi_i^{-1}(W_i)$. This means $Z=\pi_i^{-1}(\pi_i(Z))$ for all $i$. But this implies $Z=X$, a contradiction.

\end{proof}
\begin{lemma}\label{10c}
    Let $n\geq1$, $m_1, m_2,..., m_n$ be positive integers, $D$ an effective Cartier divisor in $X=\prod_{i=1}^n\mathbb{P}^{m_i}$ of multidegree a vector of $1$'s and $0$'s. Then the following holds:
    \begin{enumerate}
        \item $D$ is connected.
        \item If $D$ is irreducible, then $D$ is normal.
    \end{enumerate}
\end{lemma}
\begin{proof}
   Without loss of generality, assume that the multidegree of $D$ is $(\underline{1}, \underline{0})$ where $\underline{1}$ is a vector of length $r$, $\underline{0}$ is a vector of length $n-r$. Let $f:X\to\prod_{i\leq r}\mathbb{P}^{m_i}$ be the projection. So, $D=f^{-1}(D_1)\cong D_1\times \prod_{r<i\leq n} \mathbb{P}^{m_i}$, where $D_1$ is an effective Cartier divisor in $\prod_{i\leq r}\mathbb{P}^{m_i}$ of multidegree $(1,1,...,1)$, hence ample. So $D_1$, hence $D$ is connected, proving $(1)$. If $D$ is irreducible, $D_1$ is also irreducible. So, $D_1$ is normal by Lemma \ref{normal}. So, $D$ is normal, proving $(2).$

\end{proof}
Now we can prove Theorem \ref{main}.

\textit{Proof of Theorem \ref{main}:}
  If $n=1$, the $(i)$ follows from {\cite[Theorem 1.3, Theorem 6.1]{cheltsov2001log}}, and $(ii)$ is clear. So, we assume $n\geq 2$. Throughout the proof, for a log pair $(W,D)$, we denote the non-klt locus of $(W,D)$ by $Nklt(W,D).$
 
     \textbf{Step 1:}  First we prove the theorem assuming that every non-klt centre of $(X,\Delta)$ has dimension $0$ or $\sum_i m_i-1$. Let $\pi:X\to\prod_{i\leq n-1}\mathbb{P}^{m_i}$ the projection. By permuting the $m_i$'s if necessary, we can assume that $\pi(\floor{\Delta})$ is not a point if $\floor{\Delta}$ is irreducible. 
     
     It suffices to show there is no non-klt centre of dimension $0$. Suppose not, let $p=(q,q')\in X$ be a non-klt centre of dimension $0$, where $q\in \prod_{i\leq n-1}\mathbb{P}^{m_i}$, $q'\in \mathbb{P}^{m_n}$. Let $H$ be a general divisor containing $q$ in $\prod_{i\leq n-1}\mathbb{P}^{m_i}$ of multidegree $(1,...,1)$.

     Let $E$ be a prime divisor over $X$ with centre $p$ such that $a(E, X, \Delta)\leq -1$. Since every non-klt centre of $(X,\Delta)$ has dimension $0$ or $\sum_i m_i-1$, there is a neighborhood $U$ of $p$ such that $(U\setminus p, \Delta|_{U\setminus p})$ is plt. Since $H$ is general, $\pi^{-1}H$ and $\floor{\Delta}$ has no common irreducible components. By {\cite[Corollary 2.35]{wilson2000birational}}, $(U\setminus p, (\Delta+\epsilon \pi^{-1}H)|_{U\setminus p})$ is plt for $0<\epsilon<<1.$ Since $p\in \pi^{-1}H$, we have $a(E, X,\Delta+\epsilon \pi^{-1}H)<-1.$ So, for $0<\eta<<\epsilon,$ we have $a(E, X,(1-\eta)\Delta+\epsilon \pi^{-1}H)<-1.$ Also, since $\floor{(1-\eta)\Delta+\epsilon \pi^{-1}H}=0$, $(U\setminus p, ((1-\eta)\Delta+\epsilon \pi^{-1}H)|_{U\setminus p})$ is klt. So, $Nklt(X,(1-\eta)\Delta+\epsilon \pi^{-1}H)\cap U=\{p\}.$

     Let $H'$ be a hyperplane in $\mathbb{P}^{m_n}$ disjoint from $q'$, and let $H_1=\prod_{i\leq n-1}\mathbb{P}^{m_i}\times H'$, a prime divisor in $X$ disjoint from $p$. Let $\Delta'=H_1+(1-\eta)\Delta+\epsilon \pi^{-1}H. $ So, $H_1\cup \{p\}\subseteq Nklt(X,\Delta')$ and $p$ is an isolated point of $Nklt(X,\Delta')$. So, $Nklt(X,\Delta')$ is disconnected. But $-(K_X+\Delta')$ has multidegree $(m_1+\eta-\epsilon,...,m_{n-1}+\eta-\epsilon, m_n-1+\eta),$ a vector of positive numbers. So, $-(K_X+\Delta')$ is ample. This contradicts Koll{\'a}r-Shokurov connectedness principle {\cite[Theorem 17.4]{kollar1992flips}}.

     \textbf{Step 2:} Now we prove $(i)$ by induction of $\sum_i m_i$. If $\sum_i m_i=1$, $(i)$ is vacuously true, as we assumed $n\geq 2$. Now suppose $\sum_i m_i\geq 2$. Suppose $Nklt(X,\Delta)\neq \varnothing$. By Step 1, we must have dim $Nklt(X,\Delta)>0$. So, there is $i$ such that the image of $Nklt(X,\Delta)$ under the projection to the $i$'th factor $\mathbb{P}^{m_i}$ has positive dimension. Without loss of generality assume $i=n$. Let $\pi:X\to \mathbb{P}^{m_n}$ be the projection. Let $H_1$ be a general hyperplane in $\mathbb{P}^{m_n}$, and $H=\pi^{-1}H_1\cong \prod_{i\leq n-1}\mathbb{P}^{m_i}\times \mathbb{P}^{m_n-1}$, $\Delta_1=\Delta|_H.$ So, $\floor{\Delta_1}=0$, and $\Delta_1$ has multidegree $(1,1,..,1)$. If at most one number among $m_1,..m_{n-1}, m_n-1$ is nonzero, then $(H,\Delta_1)$ is klt as Theorem \ref{main} is true for $n=1.$ If at least two numbers among $m_1,..,m_{n-1}, m_n-1$ are nonzero, then $(H,\Delta_1)$ is klt by induction hypothesis. So, in any case, $(H,\Delta_1)$ is klt.

     By {\cite[Theorem 7.5]{kollar1997singularities}}, $(X, H+\Delta)$ is plt near $H$, so $(X,\Delta)$ is plt near $H$. Since $\floor{\Delta}=0$, $(X,\Delta)$ is klt near $H$. So, $Nklt(X,\Delta)\cap H=\varnothing$. This means $H_1\cap \pi(Nklt(X,\Delta))=\varnothing$. But this is impossible, as dim $\pi(Nklt(X,\Delta))>0.$ This contradiction shows $(X,\Delta)$ is klt. So, by induction, proof of $(i)$ is complete.

     \textbf{Step 3:} Now we prove $(ii)$ by induction of $\sum_i m_i$. If $\sum_i m_i=1$, $(ii)$ is vacuously true, as we assumed $n\geq 2$. Now suppose $\sum_i m_i\geq 2$. Write $\Delta=D_1+\Delta_1$, where $D_1$ is a prime divisor, $\floor{\Delta_1}=0,$ $D_1$ is not a component of $\Delta_1.$ If all non-klt centres of $(X,\Delta)$ has dimension $0$ or $\sum_i m_i-1,$ we are done by Step 1. So suppose there is a non-klt centre $Z$ of $(X,\Delta)$ such that $0<$ dim $Z<\sum_i m_i-1$. Without loss of generality, assume dim $\pi(Z)>0$, where $\pi:X\to \mathbb{P}^{m_n}$ is the projection. Note that $(X, \Delta_1)$ is klt. The reason is as follows: note that $|D_1|$ is base-point free, choose $D_2\in |D_1|$ general. Then $\floor{\Delta_1+\frac{1}{2}D_1+ \frac{1}{2}D_2}=0,$ $\Delta_1+\frac{1}{2}D_1+ \frac{1}{2}D_2$ has multidegree $(1,1,..1).$ By $(i)$, $(X, \Delta_1+\frac{1}{2}D_1+ \frac{1}{2}D_2)$ is klt, so $(X, \Delta_1)$ is klt. 
     
     This implies $Z\subset D_1.$ In particular, dim $\pi(D_1)>0$. 

     Now let $H_1$ be a general hyperplane in $\mathbb{P}^{m_n}$, $H=\pi^{-1}(H_1)\cong \prod_{i\leq n-1}\mathbb{P}^{m_i}\times \mathbb{P}^{m_n-1}, \Delta'=\Delta|_H.$ So, $\Delta'$ has multidegree $(1,1,..,1)$, and $\floor{\Delta'}=D_1\cap H$.

     \textit{Claim:} $D_1\cap H$ is irreducible.
     \begin{proof}
     If $m_n\geq 2$, then the claim follows by {\cite[Theorem 2.3]{das2019log}}. So suppose $m_n=1$. The multidegree of $D_1$ is a vector of $1$'s and $0$'s, so by Corollary \ref{10c}(2), $D_1$ is normal. For each $x\in \mathbb{P}^{1}$, $D_1\cap \pi^{-1}(x)$ is either $\prod_{i\leq n-1}\mathbb{P}^{m_i}$, or an effective Cartier divisor in $\prod_{i\leq n-1}\mathbb{P}^{m_i}$ of multidegree a vector of $1$'s and $0$'s, hence is connected by Lemma \ref{10c} (1). So, $\pi|_{D_1}:D_1\to \mathbb{P}^{1}$ is a contraction of normal varieties, so the general fibre of $\pi|_{D_1}$ is irreducible. So, $D_1\cap H$ is irreducible.
     \end{proof}

       If at most one number among $m_1,..,m_{n-1}, m_n-1$ is nonzero, then $(H,\Delta')$ is plt as Theorem \ref{main} is true for $n=1.$ If at least two numbers among $m_1,..m_{n-1}, m_n-1$ are nonzero, then $(H,\Delta')$ is plt by induction hypothesis. So, in any case, $(H,\Delta')$ is plt. Note that dim $(Z\cap H)=$ dim $Z-1<$ dim $D_1-1=$ dim $(D_1\cap H)$.

       Since dim $\pi(Z)>0$, we have $Z\cap H\neq \varnothing$. Let $W$ be an irreducible component of $Z\cap H.$ Now the following claim contradicts the fact that $(H, \Delta')$ is plt, hence completes the proof of $(ii)$ by induction.

     \textit{Claim:} $W$ is a non-klt centre of $(H, \Delta').$
     \begin{proof}
     Let $Y\xrightarrow{f} X$ be a birational map, $Y$ smooth projective, $E$ an $f$-exceptional divisor in $Y$ such that $a=a(E,X,\Delta)\leq -1$, and $f(E)=Z$. Let $p=\pi\circ f:Y\to \mathbb{P}^{m_n}$. Let $Y_1=p^{-1}(H_1)=f^{-1}(H).$ By Bertini, $Y_1$ is smooth. If $m_n\geq 2,$ then by {\cite[Theorem 2.3]{das2019log}}, $Y_1$ is irreducible. If $m_n=1$, then $p$ is a contraction of normal varieties, so by $Y_1$, being a general fibre of $p$, is irreducible. So, in any case, $Y_1$ is a smooth projective variety, and $f_1=f|_{Y_1}: Y_1\to H$ is birational.

     Now the proof is along the same lines as the proof of inversion of adjunction. We can write $$K_Y+ f_*^{-1}(\Delta)\sim f^*(K_X + \Delta) + aE +G,$$ where $G$ is $f$-exceptional and $E$ is not a component of $G$. Since $Y_1=f^{-1}(H)$, we have $$K_Y+ Y_1+ f_*^{-1}(\Delta)\sim f^*(K_X + H + \Delta) + aE +G.$$ Restricting to $Y_1$ and using adjunction, we get $$K_{Y_1}+ f_*^{-1}(\Delta')\sim f_1^*(K_H + \Delta') + aE_1 +G_1,$$ where $E_1= E\cap Y_1$, $G_1$ is $f_1$-exceptional divisor in $Y_1$. Note that $E_1$ and $G_1$ has no common irreducible components, as $H_1$ is general. Note that $f_1(E_1)=Z\cap H.$ If $E_2$ is an irreducible component of $E_1$ dominating $W$, then we have $a(E_2, H, \Delta')=a\leq -1$. So, $W$ is a non-klt centre of $(H, \Delta').$
     \end{proof}
 
\begin{corollary}\label{rat}
    Let $n\geq1$, $m_1, m_2,..., m_n$ be positive integers, $X=\prod_{i=1}^n\mathbb{P}^{m_i}$. Let $D$ be an irreducible divisor in $X$ of multidegree $(1,1,...,1).$ Then $D$ has rational singularities.
\end{corollary}
\begin{proof}
    By Theorem \ref{main} and \ref{10c}, $D$ is normal and $(X, D)$ is plt. So, by inversion of adjunction ({\cite[Theorem 7.5]{kollar1997singularities}}), $D$ is klt. So, $D$  has rational singularities by {\cite[Theorem 5.22]{wilson2000birational}}.
\end{proof}

Now we recover {\cite[Theorem 1.1]{bath2024singularities}}. For a positive integer $n$, let $[n]=\{1,2,...,n\}$. For a subset $I\subset [n]$, let $\underline{X}^I$ be the monomial $\prod_{i\in I} X_i$ in $\mathbb{C}[X_1,...,X_n].$ We say a polynomial $f\in \mathbb{C}[X_1,...,X_n]$ is \textit{square-free} if $f$ can be written as $\sum_{I\subset[n]} c_I \underline{X}^I$ for some $c_I\in \mathbb{C}$.
 \begin{corollary}\label{squarefree}
 Let $n\geq 1$, $Z$ a hypersurface in $\mathbb{A}^n$ defined by an irreducible square-free polynomial $f\in \mathbb{C}[X_1,...,X_n].$ Then $Z$ is normal and has rational singularities.
 \end{corollary}
 \begin{proof}
     Let us say that the $i$'th copy of $(\mathbb{P}^1)^n$ has homogeneous coordinates $[X_i:Y_i]$. We consider $\mathbb{A}^n$ as the open subset of $(\mathbb{P}^1)^n$ defined by $Y_1Y_2...Y_n\neq 0$. Let $f(X_1,...,X_n)=\sum_{I\subset[n]} c_I \underline{X}^I$. We assume without loss of generality that $f$ involves each variable $X_i$, otherwise $Z$ is pullback by a projection, of an irreducible hypersurface in $\mathbb{A}^{n-1}$ defined by a square-free polynomial, and we will be done by induction. In other words, $\cup_{I\subset[n],c_I\neq0} I=[n].$

     Let $F(\underline{X},\underline{Y})=\sum_I c_I \underline{X}^I \underline{Y}^{I^c}\in k[\underline{X},\underline{Y}].$ Here $I^c$ denotes the complement of the subset $I$ on $[n].$ Let $D\hookrightarrow (\mathbb{P}^1)^n$ be the zero-locus of $F$, so $D$ has multidegree $(1,1,...,1).$ Also, $D\cap \mathbb{A}^n=Z$. We know that $Z$ is irreducible. If $D$ is not irreducible, $D$ would contain an irreducible component of $(\mathbb{P}^1)^n\setminus\mathbb{A}^n$, so $Y_i|F$ for some $i$. So, $i\in I^c$ for all $I\subset[n]$ with $c_I\neq 0$. But this contradicts our assumption that $\cup_{I\subset[n],c_I\neq0} I=[n].$ So, $D$ is irreducible (hence $D=\Bar{Z}$). By Lemma \ref{normal} and Corollary \ref{rat}, $D$ is normal and has rational singularities. So, $Z=D\cap \mathbb{A}^n$ is normal and has rational singularities.
 \end{proof}
 \begin{corollary}\label{lct}
     Let $n,d \geq 1,$, $m_1, m_2,..., m_n$ be positive integers, $X=\prod_{i=1}^n\mathbb{P}^{m_i}$. Let $D$ be a prime divisor in $X$ of multidegree $(d,d,...,d).$ Then the log canonical threshold of $(X,D)$ satisfies $lct(X,D)\geq 1/d$. If $d\geq 2$, the inequality is strict.
 \end{corollary}
 \begin{proof}
     If $d=1$ we have $(X,D)$ is plt by Theorem \ref{main}, hence lc. So the claim follows. For $d\geq 2,$ since $\floor{\frac{1}{d} D}=0$, $(X,\frac{1}{d} D)$ is klt by Theorem \ref{main}. So, $lct(X,D)>1/d$.
 \end{proof}
 \section{Acknowledgement}
 I thank Prof. János Kollár for giving valuable ideas.
\printbibliography

@article{cheltsov2001log,
  title={Log canonical thresholds on hypersurfaces},
  author={Cheltsov, IA},
  journal={Sbornic: Mathematics C/c of Matematicheskii Sbornik},
  volume={192},
  number={11/12},
  pages={1241--1258},
  year={2001},
  publisher={LONDON MATHEMATICAL SOCIETY}
}

@book{wilson2000birational,
  author = {J. Kollár and M. Shigefumi},
  title = {Birational Geometry of Algebraic Varieties},
  publisher = {Cambridge University Press},
  year = {1998}
}

@inproceedings{kollar1997singularities,
  title={Singularities of pairs},
  author={Koll{\'a}r, J{\'a}nos},
  booktitle={Proceedings of Symposia in Pure Mathematics},
  volume={62},
  pages={221--288},
  year={1997},
  organization={American Mathematical Society}
}

@article{bath2024singularities,
  title={Singularities of square-free polynomials},
  author={Bath, Daniel and Musta{\c{t}}{\u{a}}, Mircea and Walther, Uli},
  journal={arXiv preprint arXiv:2412.11309},
  year={2024}
}

@article{das2019log,
  title={On the log minimal model program for $3 $-folds over imperfect fields of characteristic $ p> 5$},
  author={Das, Omprokash and Waldron, Joe},
  journal={arXiv preprint arXiv:1911.04394},
  year={2019}
}

@article{kollar1992flips,
  title={Flips and abundance for algebraic threefolds},
  author={Koll{\'a}r, J{\'a}nos},
  journal={Ast{\'e}risque},
  volume={211},
  year={1992},
  publisher={Soc. Math. France.}
}
\vspace{40pt}
\begin{flushleft}
{\scshape Fine Hall, Princeton University, Princeton, NJ 700108, USA}.

{\fontfamily{cmtt}\selectfont
\textit{Email address: ss6663@princeton.edu} }
\end{flushleft}
\end{document}